\documentclass{amsart}
\usepackage{amsmath,amsthm,amssymb,amsfonts,amscd,mathrsfs,mathtools,enumerate}
\usepackage{tikz,graphicx}
\usepackage[all]{xy}
\usepackage[hyphens]{url}
\usepackage{hyperref}
\usepackage{caption}
\usetikzlibrary{decorations.markings}
\usetikzlibrary{matrix,arrows}

\theoremstyle{plain}
    \newtheorem{thm}{Theorem}[section]
    \newtheorem{lem}[thm]   {Lemma}
    \newtheorem{cor}[thm]   {Corollary}
    \newtheorem{prop}[thm]  {Proposition}
    
\theoremstyle{definition}
    \newtheorem{defn}[thm]  {Definition}

    \newtheorem{ex}[thm]{Example}
    \newtheorem{rem}[thm]{Remark}
  
  \newtheorem*{thm*}{Theorem}
  \newtheorem*{prop*}{Proposition}

\def\Hom{\mathrm{Hom}}

\def\cat{\mathsf{cat}}
\def\secat{\mathsf{secat}}
\def\genus{\mathsf{genus}}

\def\dim{\mathrm{dim}}
\def\ker{\mathrm{ker}}

\def\tto{\twoheadrightarrow}

\newcommand{\be}{\begin{enumerate}}
\newcommand{\ee}{\end{enumerate}}
\newcommand{\R}{\mathbb{R}}
\newcommand{\Z}{\mathbb{Z}}
\newcommand{\C}{\mathbb{C}}

\newcommand{\TC}{{\sf TC}}

\newcommand{\cld}{{\sf cd}}

\begin{document}

\title[Parametrised TC of epimorphisms]{Parametrised topological complexity of group epimorphisms}

\author{Mark Grant}

\address{Institute of Mathematics,
Fraser Noble Building,
University of Aberdeen,
Aberdeen AB24 3UE,
UK}

\email{mark.grant@abdn.ac.uk}

\date{\today}


\begin{abstract} We show that the parametrised topological complexity of Cohen, Farber and Weinberger gives an invariant of group epimorphisms. We extend various bounds for the topological complexity of groups to obtain bounds for the parametrised topological complexity of epimorphisms. Several applications are given, including an alternative computation of the parametrised topological complexity of the planar Fadell--Neuwirth fibrations which avoids calculations involving cup products. We also prove a homotopy invariance result for parametrised topological complexity of fibrations over different bases. 
\end{abstract}

%\thanks{}

\maketitle
\section{Introduction}\label{sec:intro}

The topological complexity $\TC(X)$ of a space $X$ is defined to be the Schwarz genus, or sectional category, of the endpoint fibration $\pi:X^I\to X\times X$ which sends a free path to its pair of endpoints. This notion due to Michael Farber \cite{Far} is of potential applicability in the field of robotics, since sections of $\pi$ correspond to motion planning algorithms for mechanical systems with $X$ as their space of configurations.

One potential obstacle to applying $\TC(X)$ to robotics problems is that it assumes that the configuration space is known in advance. In order to address this, Cohen, Farber and Weinberger \cite{CFW1,CFW2} have introduced a parametrised version of topological complexity, which models motion planning problems for which the configuration space varies within a fixed homotopy type according to some space of parameters. Briefly, let $p:E\to B$ be a surjective fibration with path-connected fibre $X$. Letting $E^I_B$ denote the space of paths in $E$ with image in a single fibre of $p$, we have a parametrised endpoint fibration $\Pi:E^I_B\to E\times_B E$ which again sends a path to its pair of endpoints. Then the parametrised topological complexity of $p:E\to B$, denoted $\TC[p:E\to B]$ (or $\TC_B(X)$ if the role of the space $X$ is to be emphasized), is defined to be the sectional category of $\Pi$. Further details will be given in Section 2 below. 

As well as generalizing the topological complexity, the parametrised topological complexity has interesting mathematical properties. The papers \cite{CFW1, CFW2} compute the parametrised topological complexity of the Fadell--Neuwirth fibrations $p:F(\R^d,m+n)\to F(\R^d,m)$, which model the motion planning problem of $n$ agents moving in Euclidean space avoiding collisions with each other and with $m$ obstacles, whose positions are \emph{a priori} unknown. These calculations show that $\TC_B(X)$ can exceed $\TC(X)$ by an arbitrary amount. The paper \cite{G-C} by Garc\'{i}a-Calcines considers this new invariant from the point of view of fibrewise topology.

Any homotopy invariant of spaces gives an invariant of discrete groups via the correspondence between groups and homotopy $1$-types. Farber \cite{FarSurvey} posed the problem of describing $\TC(\pi):=\TC(K(\pi,1))$ in terms of other algebraic invariants of the group $\pi$. This problem has stimulated a great deal of research (see \cite{FGLO, FarMes,Dranish} for recent examples) but remains unanswered. By contrast, for the related invariant Lusternik--Schnirelmann category the corresponding problem has a simple solution after ground-breaking work of Eilenberg and Ganea \cite{EG}, Stallings \cite{Sta} and Swan \cite{Swa}: one has $\cat(\pi):=\cat(K(\pi,1))=\cld(\pi)$, the cohomological dimension of the group.

In Section 3 we define the parametrised topological complexity $\TC[\rho:G\tto Q]$ of an epimorphism of discrete groups, and show that it agrees with the Cohen--Farber--Weinberger definition applied to aspherical fibrations. Although our definition
\[
\TC[\rho:G\tto Q]:=\secat(\Delta:G\to G\times_Q G),
\]
where $\Delta(g)=(g,g)$ is the diagonal homomorphism into the fibred product, makes sense for arbitrary homomorphisms, we restrict our attention to epimorphisms since these correspond to fibrations of aspherical spaces with path-connected fibre. 

Since the parametrised topological complexity of epimorphisms generalises the topological complexity of groups, we do not expect or seek a purely algebraic description. We show here however that certain bounds for the topological complexity of groups given by the author, Lupton and Oprea \cite{GLO} and the author \cite{Grant} admit generalizations to the parametrised setting. These bounds for $\TC[\rho:G\tto Q]$ depend only on the cohomological dimensions of various auxiliary groups, and do not require the calculation of cup products. In particular, in Sections 4 and 5 respectively we prove:

\begin{thm*}
Let $\rho:G\tto Q$ be an epimorphism with kernel $N$. Given subgroups $A,B\leq G$ such that $gAg^{-1}\cap B=\{1\}$ for all $g\in N$, we have
\[
\cld(A\times_Q B)\leq \TC[\rho:G\tto Q],
\]
where $A\times_Q B=\{(a,b)\in A\times B \mid \rho(a)=\rho(b)\}$ is the fibred product.
\end{thm*}

\begin{thm*}  
Let $\rho:G\tto Q$ be an epimorphism with kernel $N$. Let $H\lhd G$ be a normal subgroup with $[H,N]=1$. Then $\Delta(H)$ is normal in $G\times_Q G$, and
\[
\TC[\rho:G\tto Q]\leq \cld\left(\frac{G\times_Q G}{\Delta(H)}\right).
\]
\end{thm*}

These bounds are employed in Section 6, together with classical results on duality groups due to Bieri and Eckmann \cite{BE}, to give a new computation of the parametrised topological complexity of the Fadell--Neuwirth fibrations in the planar case which appears to be more conceptual that the computation in \cite{CFW2} involving cup-lengths.

Our methods can also be used to give an easy example of a fibration for which $\TC(X)=1$ and $\TC_B(X)=\infty$ (Example~\ref{ex:infinite}), and to prove that if $N$ is central in $G$ then $\TC[\rho:G\tto Q]=\cld(N)$ (Corollary~\ref{cor:central}). We also prove the following homotopy invariance for parametrised topological complexity (Proposition~\ref{prop:hinv}), which should be of independent interest.

\begin{prop*}
Let $p:E\to B$ and $p':E'\to B'$ be fibrations.  Assume given a commutative diagram 
\[
\xymatrix{
E \ar[r]^{h} \ar[d]_p & E'\ar[d]^{p'} \\
B \ar[r]^{\bar{h}} & B'
}
\]
in which $h$ and $\bar{h}$ are homotopy equivalences. Then 
\[
\TC[p:E\to B]=\TC[p':E'\to B'].
\]
\end{prop*}

The author would like to thank Dan Cohen and Michael Farber for stimulating conversations.

\section{Sectional category and parametrised topological complexity}

In this section we recall some material about sectional category and parametrised topological complexity. We claim no originality, with the possible exception of Proposition~\ref{prop:hinv} which is a mild generalisation of a result in \cite{CFW1}.

Recall that a \emph{(Hurewicz) fibration} is a map having the homotopy lifting property with respect to all spaces, while a \emph{Serre fibration} is a map having the homotopy lifting property with respect to CW-complexes.
Schwarz \cite{Schwarz} defined the \emph{genus} of a fibration $p:E\to B$, denoted $\genus(p)$, to be the smallest $k$ such that $B$ admits a cover by open sets $U_0,\ldots, U_k$, each of which admits a local section, i.e. a map $s_i:U_i\to E$ such that $p\circ s_i$ equals the inclusion $j_i:U_i\hookrightarrow B$. This notion was subsequently generalised to give an invariant of arbitrary maps (see \cite{AS} and references therein). The \emph{sectional category} of a map $p:E\to B$, denoted $\secat(p)$, is defined to be the smallest $k$ such that $B$ admits a cover by open sets $U_0,\ldots, U_k$, each of which admits a local homotopy section, i.e. a map $\sigma_i:U_i\to E$ such that $p\circ \sigma_i\simeq j_i:U_i\hookrightarrow B$.

The following lemma collects some basic facts about the genus and sectional category. Proofs can be found in \cite{AS} or \cite{Schwarz}, or supplied by the reader.

\begin{lem}\label{lem:secatfacts}
Let $p:E\to B$ be a map.
\begin{itemize}
\item[(a)] If $p:E\to B$ is a fibration, then $\secat(p)=\genus(p)$.
\item[(b)] If $p:E\to B$ is a fibration and $q:D\to A$ is the pull-back of $p$ along some map $\alpha:A\to B$, then $\secat(q)\leq\secat(p)$.
\item[(c)] If $p':E\to B$ is homotopic to $p$, then $\secat(p')=\secat(p)$.
\item[(d)] Let $h:E'\to E$ be a homotopy equivalence. Then $\secat(p\circ h)=\secat(p)$.
\item[(e)] Let $\alpha:B\to B'$ be a homotopy equivalence. Then $\secat(\alpha\circ p)=\secat(p)$.
\end{itemize}
\end{lem}

Two fundamental examples of sectional category are the (Lusternik--Schnirelmann) category and topological complexity. Let $Y$ be a path-connected space with base point $y_0$. The \emph{category} of $Y$ may be defined by
\[
\cat(Y)=\secat(\{y_0\}\hookrightarrow Y)=\secat(p:P_0Y\to Y),
\]
where $P_0Y=\{\gamma:I\to Y\mid \gamma(0)=y_0\}$ is the based path space and $p:\gamma\mapsto\gamma(1)$ is the evaluation fibration. The \emph{topological complexity} of $Y$ may be defined by
\[
\TC(Y)=\secat(d:Y\to Y\times Y)=\secat(\pi:Y^I\to Y\times Y),
\]
where $d:Y\to Y\times Y$ is the diagonal map, $Y^I=\{\gamma:I\to Y\}$ is the free path space and $\pi:\gamma\mapsto\big(\gamma(0),\gamma(1)\big)$ is the endpoint fibration.   

The following lemma, implicit in \cite{Grant}, can be used to bound the sectional category of a map by finding a suitable fibering of its codomain.

\begin{lem}\label{lem:secatfibering}
Let $p: E\to B$ be a map. Suppose we have a homotopy commutative diagram 
\[
\xymatrix{ 
 & E\ar[d]^p & \\
F\ar[r]^{\iota} \ar[ur] & B \ar[r]^q & Y
}
\]
in which $Y$ is connected and the row is a fibration sequence. Then $\secat(p)\leq \cat(Y)$.
\end{lem}

\begin{proof}
Let $U$ be an open set in $Y$ such that the inclusion $U\hookrightarrow Y$ factors through $\{y_0\}\hookrightarrow Y$ up to homotopy. The homotopy lifting property for $q$ implies that the inclusion $q^{-1}(U)\hookrightarrow B$ factors through $\iota:F\to B$, and hence through $p:E\to B$, up to homotopy.
\end{proof}

Given a map $p:E\to B$, the \emph{parametrised endpoint map} is defined by
$$\Pi:E^I_B\to E\times_B E, \qquad \Pi(\gamma)=\big(\gamma(0),\gamma(1)\big)$$
where $E^I_B=\{\gamma: I\to E \mid p(\gamma(t))=p(\gamma(0))\mbox{ for all }t\in I\}$ is the space of paths in $E$ contained in a single fibre of $p$, and $E\times_B E:=\{(e_1,e_2)\in E\times E \mid p(e_1)=p(e_2)\}$ is the fibred product.

The parametrised endpoint map $\Pi:E^I_B\to E\times_B E$ is a Hurewicz fibration whenever $p$ is a Hurewicz fibration; a proof is supplied in the Appendix to \cite{CFW2}.

\begin{defn}[\cite{CFW1, CFW2}]
Let $p:E\to B$ be a fibration. The {\em parametrised topological complexity} of $p$ is defined to be $$\TC[p:E\to B]:=\secat(\Pi),$$ the sectional category of the parametrised endpoint fibration $\Pi: E^I_B\to E\times_B E$. 

\end{defn}
The following lemma describes the behaviour of parametrised topological complexity under taking pull-backs, and is a generalization of \cite[(4.3)]{CFW1}; see also \cite[Corollary 15]{G-C}.

\begin{lem}\label{lem:pullbacks}
Let $q:D\to A$ be the pull-back of the fibration $p:E\to B$ under a map $\alpha:A\to B$. Then $$\TC[q:D\to A]\leq \TC[p:E\to B].$$
\end{lem}

\begin{proof}
We have $D=\{(a,e)\in A\times E \mid \alpha(a)=p(e)\}$, and a commuting diagram
\[
\xymatrix{
D \ar[d]_{q} \ar[r]^{\tilde\alpha} & E \ar[d]^p \\
A \ar[r]^\alpha & B
}
\]
in which $\tilde\alpha(a,e)=e$ and $q(a,e)=a$. Observe that $\tilde\alpha$ induces a map $D\times_A D\to E\times_B E$. The reader can verify that the pull-back of $\Pi:E^I_B\to E\times_B E$ under this map is homeomorphic as a fibration to $\Phi:D^I_A\to D\times_A D$, the parametrised endpoint fibration associated to $q$. Hence by Lemma~\ref{lem:secatfacts}(b),
\[
\TC[q:D\to A]=\secat(\Phi)\leq \secat(\Pi)=\TC[p:E\to B].
\]
\end{proof}

We now turn to the homotopy invariance of $\TC[p:E\to B]$. Recall that a \emph{fibrewise map} from a fibration $p:E\to B$ to another fibration $p':E'\to B$ is a map $h:E\to E'$ such that $p'(h(e))=p(e)$ for all $e\in E$.
%the triangle 
%\[
%\xymatrix{
% E \ar[rr]^h \ar[rd]_p & & E' \ar[ld]^{p'} \\
%  & B & 
%  }
%  \]
%  commutes. 
A \emph{fibrewise homotopy} is a map $H:E\times I\to E'$ such that %the triangle
%\[
%\xymatrix{
% E\times I \ar[rr]^H \ar[rd]_{p%\circ pr_E} & & E' \ar[ld]^{p'} \\
%  & B & 
%  }
%  \]
%  commutes, 
$p'(H(e,t))=p(e)$ for all $e\in E$, $t\in I$, so that for each $t\in I$ the map $h_t:=H(-,t):E\to E'$ is a fibrewise map.  Then $H$ is a fibrewise homotopy from $h_0$ to $h_1$. Two fibrations $p:E\to B$ and $p':E'\to B$ are \emph{fibrewise homotopy equivalent} if there are fibrewise maps $h:E\to E'$ and $i:E'\to E$ such that $i\circ h:E\to E$ is fibrewise homotopic to ${\rm Id}_E:E\to E$ and $h\circ i:E'\to E'$ is fibrewise homotopic to ${\rm Id}_{E'}:E'\to E'$. It is well known that if fibrewise map $h:E\to E'$ is a homotopy equivalence, then it is a fibrewise homotopy equivalence.
  
\begin{prop}[{\cite[Proposition 5.1]{CFW1}}]  
If $p:E\to B$ and $p':E'\to B$ are fibrewise homotopy equivalent fibrations, then $\TC[p:E\to B] = \TC[p':E'\to B]$.
\end{prop}

We will generalise this result somewhat in order to compare fibrations over different bases. Let $p:E\to B$ and $p':E'\to B'$ be fibrations. A \emph{fibre-preserving map} from $p$ to $p'$ is a pair of maps $h:E\to E'$ and $\bar{h}:B\to B'$ such that the following diagram commutes:
\[
\xymatrix{
E \ar[r]^{h} \ar[d]_p & E'\ar[d]^{p'} \\
B \ar[r]^{\bar{h}} & B'.
}
\]
We will denote such a pair by $(h,\bar{h}):p\to p'$. It is obvious that fibre-preserving maps are the morphisms in a category whose objects are the fibrations.
 
A \emph{fibre-preserving homotopy} is a pair of homotopies $H:E\times I\to E'$ and $\bar{H}:B\times I\to B'$ such that the following diagram commutes:
\[
\xymatrix{
E\times I \ar[r]^{H} \ar[d]_{p\times\mathrm{Id}_I} & E' \ar[d]^{p'} \\
B\times I \ar[r]^{\bar{H}} & B'.
}
\]
In other words, for each $t\in I$ the pair $h_t:=H(-,t):E\to E'$ and $\bar{h}_t:=\bar{H}(-,t):B\to B'$ form a fibre-preserving map. Then the pair $(H,\bar{H})$ is a fibre-preserving homotopy from $(h_0,\bar{h}_0):p\to p'$ to $(h_1,\bar{h}_1):p\to p'$. Two fibrations $p:E\to B$ and $p':E'\to B'$ are said to be \emph{fibre-preserving homotopy equivalent} if there are fibre-preserving maps $(h,\bar{h}):p\to p'$ and $(i,\bar{i}):p'\to p$ such that $(i\circ h,\bar{i}\circ\bar{h})$ is fibre-preserving homotopic to $({\rm Id}_{E},{\rm Id}_{B})$ and $(h\circ i,\bar{h}\circ\bar{i})$ is fibre-preserving homotopic to $({\rm Id}_{E'},{\rm Id}_{B'})$.

\begin{prop}\label{prop:hinv}
Let $p:E\to B$ and $p':E'\to B'$ be fibrations. Assume given a commutative diagram 
\[
\xymatrix{
E \ar[r]^{h} \ar[d]_p & E'\ar[d]^{p'} \\
B \ar[r]^{\bar{h}} & B'
}
\]
in which $h$ and $\bar{h}$ are homotopy equivalences. Then 
\[
\TC[p:E\to B]=\TC[p':E'\to B'].
\]
\end{prop}

\begin{proof}
According to the Proposition on page 53 of \cite{May}, the pair $(h,\bar{h})$ is in fact a fibre-preserving homotopy equivalence. It therefore suffices to show that fibre-preserving homotopy equivalent fibrations have the same parametrised topological complexity.

Observe that a fibre-preserving homotopy $(H,\bar{H})$ from $p$ to $p'$ induces a fibre-preserving homotopy
\[
\xymatrix{
E^I_B\times I \ar[r]^{F} \ar[d]_{\Pi\times\mathrm{Id}_I} & (E')^I_{B'} \ar[d]^{\Pi'} \\
(E\times_B E)\times I \ar[r]^{\bar{F}} & E'\times_{B'} E'
}
\]
between parametrised endpoint fibrations. Explicitly, the homotopy $F$ is defined by 
\[
F(\gamma,t)(s)=H(\gamma(s),t),\qquad \gamma\in E^I_B, \, s,t\in I
\]
and the homotopy $\bar{F}$ is defined by
\[
\bar{F}(e_1,e_2,t)=\big(H(e_1,t), H(e_2,t)\big),\qquad (e_1,e_2)\in E\times_B E,\, t\in I.
\]
 It follows that if $p$ and $p'$ are fibre-preserving homotopy equivalent, then there is a fibre-preserving map
 \[
 \xymatrix{
 E^I_B \ar[r]^h \ar[d]_\Pi & (E')^I_{B'} \ar[d]^{\Pi'} \\
 E\times_B E \ar[r]^{\bar{h}} & E'\times_{B'} E'
 }
 \]
 between their parametrised endpoint fibrations such that both $h$ and $\bar{h}$ are homotopy equivalences. Then
 \[
 \TC[p:E\to B]=\secat(\Pi)=\secat(\Pi')=\TC[p':E'\to B']
 \]
 by parts (d) and (e) of Lemma~\ref{lem:secatfacts}. 
\end{proof}

\section{Parametrised TC of group epimorphisms}

Recall that the topological complexity of a discrete group $\pi$ is defined by $\TC(\pi):=\TC(K(\pi,1))$, where $K(\pi,1)$ stands for any \emph{Eilenberg--MacLane space} for $\pi$ (a connected space with fundamental group $\pi$ and trivial higher homotopy groups). Such a space may be taken to be a CW-complex, which then is unique up to homotopy equivalence. This follows from the well-known bijection
\[
\left[K(\pi,1),K(\pi',1)\right]_* \cong \Hom(\pi,\pi')
\]
between pointed homotopy classes of pointed Eilenberg--MacLane CW-complexes and homomorphisms of groups. Therefore $\TC(\pi)$ is well-defined. 

In this section we define the parametrised topological complexity $\TC[\rho:G\tto Q]$ of an epimorphism $\rho:G\tto Q$ of discrete groups. We first define the sectional category of group homomorphisms. It will be useful to introduce the following terminology.

\begin{defn}
We say that a pointed map $f:X\to Y$ \emph{realizes} the group homomorphism $\varphi: G\to Q$ if $X$ is a $K(G,1)$ space, $Y$ is a $K(Q,1)$ space, and $f$ induces $\varphi$ on fundamental groups.
\end{defn}

\begin{lem}\label{lem:realize}
Any homomorphism $\varphi:G\to Q$ of discrete groups may be realized by a pointed map. Furthermore, any two such  maps have the same sectional category. 
\end{lem}

\begin{proof}
The first statement is clear, due to the existence of Eilenberg--MacLane spaces and the bijection 
\[
\left[K(G,1),K(Q,1)\right]_* \cong \Hom(G,Q).
\]

Let $f:X\to Y$ and $f':X'\to Y'$ be pointed maps both of which realize $\varphi:G\to Q$. Then there are pointed homotopy equivalences $h:X\to X'$ (realizing the identity $G\to G$) and $\alpha:Y\to Y'$ (realizing the identity $Q\to Q$) such that the following diagram commutes up to pointed homotopy:
\[
\xymatrix{
X\ar[r]^h \ar[d]_f  &X' \ar[d]^{f'} \\
Y \ar[r]^\alpha &  Y'
}
\]
Using parts (c), (d) and (e) of Lemma~\ref{lem:secatfacts} we then have
\[
\secat(f)=\secat(\alpha\circ f)=\secat(f'\circ h)=\secat(f').
\]
\end{proof}
 The above lemma allows us to make the following definition (compare \cite[\S 1]{BCE} where the focus is on group \emph{mono}morphisms).
 \begin{defn}
 The sectional category of a group homomorphism $\varphi:G\to Q$ is defined to be $$\secat(\varphi):=\secat(f)$$ for any map $f$ which realizes $\varphi$.
 \end{defn}

For example, we have $\cat(\pi):=\cat(K(\pi,1))=\secat(\iota:1\to \pi)$ where $\iota$ denotes the inclusion of the identity element, and $\TC(\pi)=\secat(\Delta:\pi\to \pi\times\pi)$ where $\Delta$ denotes the diagonal homomorphism.

Given homomorphisms $\alpha:A\to Q$ and $\beta: B\to Q$, we may form their fibred product
\[
A\times_Q B = \{(a,b)\in A\times B \mid \alpha(a)=\beta(b)\}.
\]
It is a subgroup of the direct product $A\times B$.

\begin{defn}\label{def:pTChomo}
Let $\varphi:G\to Q$ be a group homomorphism. The \emph{parametrised topological complexity} of $\varphi$ is defined by
\[
\TC[\varphi:G\to Q]:=\secat(\Delta:G\to G\times_Q G),
\]
where $\Delta(g)=(g,g)$ is the diagonal homomorphism to the fibred product.
\end{defn}

Although the above definition can be made for arbitrary group homomorphisms, our focus will be on epimorphisms $\rho:G\tto Q$, for the following reason. Recall that a fibration $p:E\to B$ is \emph{$0$-connected} if it is surjective and has path-connected fibres. While any homomorphism can be realised by a surjective fibration (use the path space construction to convert a realizing map to a fibration), it is precisely the epimorphisms which can be realised by $0$-connected fibrations.
 
\begin{prop}\label{prop:diagonal}
Let $p:E\to B$ be a $0$-connected fibration which realizes the epimorphism $\rho:G\tto Q$. Then 
\[
\TC[p:E\to B]=\TC[\rho:G\tto Q].
\]
\end{prop}

\begin{proof}
By definition $\TC[p:E\to B]=\secat(\Pi:E^I_B\to E\times_B E)$. The map $h:E\to E^I_B$ which embeds $E$ as the constant paths has as homotopy inverse the map $m:E^I_B\to E$ which sends a path to its midpoint. Furthermore, the following diagram commutes, where $d:E\to E\times_B E$ is the diagonal embedding:
\[
\xymatrix{
E \ar[rr]^{h} \ar[dr]_{d} &  & E^I_B \ar[ld]^{\Pi} \\
 & E\times_B E
 }
 \]
By Lemma~\ref{lem:secatfacts}(d), we therefore have $\TC[p:E\to B]=\secat(d:E\to E\times_B E)$. The proof will be complete once we show that $d:E\to E\times_B E$ realizes the diagonal homomorphism $\Delta:G\to G\times_Q G$.

By assumption $E$ is a $K(G,1)$ and $B$ is a $K(Q,1)$. Let $X$ be the fibre of $p:E\to B$, which by assumption is path-connected. Note that $E\times_B E$ fibres over $E$ with fibre $X$, hence is path-connected. We examine the Mayer--Vietoris sequence associated to the fibre sequence $\Omega B\to E\times_B E \to E\times E$ (see for instance \cite[Corollary 2.2.3]{MayPonto}), in which the map (of pointed sets) $\pi_1(E)\times\pi_1(E)\to \pi_1(B)$ is given by $(g,h)\mapsto\rho(g)\cdot \rho(h)^{-1}$:
\[
\begin{tikzpicture}[descr/.style={fill=white,inner sep=1.5pt}]
        \matrix (m) [
            matrix of math nodes,
            row sep=1em,
            column sep=2.5em,
            text height=1.5ex, text depth=0.25ex
        ]
        { \cdots & \pi_i(E\times_B E) & \pi_i(E)\times\pi_i(E) & \pi_i(B) \\
                 & \mbox{}         &                 & \mbox{}         \\
            & \pi_2(E\times_B E) & \pi_2(E)\times\pi_2(E) & \pi_2(B) \\
            & \pi_1(E\times_B E) & \pi_1(E)\times\pi_1(E) & \pi_1(B) \\
               & \pi_0(E\times_B E) & \pi_0(E)\times\pi_0(E) & \pi_0(B)\\
        };

        \path[overlay,->, font=\scriptsize,>=latex]
        (m-1-1) edge (m-1-2)
        (m-1-2) edge (m-1-3)
        (m-1-3) edge (m-1-4)
        (m-1-4)  edge[out=355,in=175,dashed] (m-3-2)
        (m-3-2) edge (m-3-3)
        (m-3-3) edge (m-3-4)
        (m-3-4) edge[out=355,in=175] node[descr,yshift=0.3ex] {$\partial$} (m-4-2)
        (m-4-2) edge (m-4-3)
        (m-4-3) edge (m-4-4)
        (m-4-4) edge[out=355,in=175] node[descr,yshift=0.3ex] {$\partial$} (m-5-2)
        (m-5-2) edge (m-5-3)
        (m-5-3) edge (m-5-4);
\end{tikzpicture}
\]

This confirms that $E\times_B E$ is a $K(G\times_Q G,1)$, and identifies the inclusion induced map $\pi_1(E\times_B E)\to \pi_1(E\times E)$ with the subgroup inclusion $G\times_Q G\hookrightarrow G\times G$. From this it follows easily that $d:E\to E\times_B E$ realizes the diagonal $\Delta:G\to G\times_Q G$. 
\end{proof}

\begin{rem}
By Proposition~\ref{prop:diagonal}, any two $0$-connected fibrations $p:E\to B$ and $p':E'\to B'$ which realize $\rho:G\tto Q$ have the same parametrised topological complexity. This also follows directly from Proposition~\ref{prop:hinv}, by an argument similar to the proof of Lemma~\ref{lem:realize}. 
\end{rem}

Let $\rho:G\tto Q$ be an epimorphism with kernel $N\lhd G$. Then the fibre $X$ of any $0$-connected fibration $p:E\to B$ which realizes $\rho$ is a $K(N,1)$ space. We may therefore introduce the alternative notation $\TC_Q(N)$ for $\TC[\rho:G\tto Q]$, although this de-emphasizes the important role played by the extension $G$.

We end this section with some trivial observations relating the parametrised topological complexity of epimorphisms with the topological complexity of groups.

\begin{ex}
Let $\rho:G\tto 1$ be the epimorphism from a group $G$ to the trivial group. Then $\TC[\rho:G\to 1]=\TC(G)$.
\end{ex} 

\begin{lem}\label{lem:fibre}
Let $N=\ker(\rho:G\tto Q)$. Then $\TC(N)\leq \TC[\rho:G\tto Q]$.
\end{lem}

\begin{proof}
Since the fibre $X$ of $p:E\to B$ is a $K(N,1)$ space, this is a special case of \cite[(4.4)]{CFW1}.
\end{proof}

\section{A lower bound}

In this section we establish a lower bound for the parametrised topological complexity of group epimorphisms, which generalizes the lower bound for the topological complexity of groups due to Grant--Lupton--Oprea \cite{GLO}. In particular, it depends only on the cohomological dimension of certain subgroups, and therefore circumvents cup-product calculations.

Recall that the cohomological dimension of a discrete group $\pi$, denoted $\cld(\pi)$, is the minimal length of a projective resolution of $\Z$ by $\Z\pi$-modules, or equivalently the largest $k$ such that $H^k(\pi;M)\neq0$ for some $\Z\pi$-module $M$. Following celebrated results of Eilenberg--Ganea \cite{EG}, Stallings \cite{Sta} and Swan \cite{Swa}, we have $\cld(\pi)=\cat(\pi)$, where the latter denotes the category of any $K(\pi,1)$ space.

\begin{thm}\label{thm:lower}
Let $\rho: G\tto Q$ be a group epimorphism with kernel $N$. Given subgroups $A,B\leq G$ such that $gAg^{-1}\cap B=\{1\}$ for all $g\in N$, we have
\[
\cld(A\times_Q B) \leq \TC[\rho:G\tto Q].
\]
Here $A\times_Q B=\{(a,b)\in A\times B \mid \rho(a)=\rho(b)\}$ is the fibred product of $\rho|_A:A\to Q$ and $\rho|_B:B\to Q$.
\end{thm}

\begin{proof}
The proof follows closely that of \cite[Theorem 1.1]{GLO}.

Let $A,B\leq G$ be subgroups as in the statement, and let $p:E\to B$ be a fibration realizing $\rho:G\tto Q$. Let $Y$ be a $K(A\times_Q B,1)$ space, and let $\psi:Y\to E\times_B E$ be a map realizing the inclusion $A\times_Q B\hookrightarrow G\times_Q G$. Form the pull-back
\[
\xymatrix{
D \ar[d]_q \ar[r] &  E^I_B \ar[d]^\Pi \\
Y \ar[r]^-\psi & E\times_B E
}
\]
and observe that $\secat(q)\leq \secat(\Pi)=\TC[\rho:G\tto Q]$ by Lemma~\ref{lem:secatfacts}(b). We will show that under the given assumptions, $\cld(A\times_Q B)=\cat(Y)\leq \secat(q)$.

Recall that the \emph{$1$-dimensional category} of a space $Y$, denoted $\cat_1(Y)$, is the smallest $k$ such that $Y$ admits a cover by $1$-categorical open sets $U_0,\ldots , U_k$. An open set $U\subseteq Y$ is \emph{$1$-categorical} if every composition $L\to U\hookrightarrow Y$ where $L$ is a CW-complex with $\dim(L)\leq 1$ is null-homotopic. It is well-known that if $Y$ is a $K(\pi,1)$ space then $\cat_1(Y)=\cat(Y)$. We are therefore reduced to showing that $\secat(q)\geq \cat_1(Y)$.

Let $U\subseteq Y$ be an open set such that the inclusion $U\hookrightarrow Y$ factors through $q:D\to Y$. We must show that $U$ is $1$-categorical, which by \cite[Lemma 5.3]{GLO} is equivalent to showing that every composition $\phi: S^1\to U\hookrightarrow Y$ is null-homotopic. Applying the functor $[S^1,-]$ given by unbased homotopy classes of loops, we have the following diagram
\[
\xymatrix{
[S^1,D] \ar[d]_{q_*} \ar[r] & [S^1,E^I_B]\ar[d]_{\Pi_*} \ar[r]^{m_*} & [S^1,E]\ar[dl]^{d_*} \\
[S^1,Y] \ar[r]^-{\psi_*} & [S^1,E\times_B E] & 
}
\]
Recall that for connected pointed spaces there is a natural bijection between $[S^1,X]$ and the set of conjugacy classes in $\pi_1(X,x_0)$. Therefore $[\phi]\in [S^1,Y]$ corresponds to some conjugacy class $[(a,b)]$ in $A\times_Q B$. Since $[\phi]$ is in the image of $q_*$, by the above diagram $\psi_*[\phi]$ is in the image of $d_*$, which in terms of conjugacy classes implies that $(a,b)$ is conjugate in $G\times_Q G$ to some element of the diagonal subgroup $\Delta(G)$. It follows that there exists $(k,\ell)\in G\times_Q G$ such that $kak^{-1}=\ell b \ell^{-1}$. But then $g:=\ell^{-1}k\in N$ conjugates $a\in A$ to $b\in B$, which under our assumptions implies that $a=b=1$, and hence that $\phi$ is null-homotopic. 
\end{proof}

\begin{rem}
Theorem~\ref{thm:lower} specializes to \cite[Theorem 1.1]{GLO} in the case of the trivial epimorphism $\rho:G\tto 1$. For general epimorphisms, we have more flexibility in choosing the subgroups $A$ and $B$ (since we only need trivial intersections under conjugation by elements of the kernel) but the conclusion may be weaker (since $A\times_Q B\leq A\times B$ implies $\cld(A\times_Q B)\leq \cld(A\times B)$.)
\end{rem} 

\begin{ex}\label{ex:infinite}
Let $G=\Z\rtimes\Z/2$ be the infinite dihedral group, expressed as the semi-direct product of the sign representation of $\Z/2$ on $\Z$, and let $\rho:G\to Q=\Z/2$ be the projection. Note that $\TC(N)=\TC(\Z)=1$. We can show that $\TC_{\Z/2}(\Z)=\infty$ using Theorem~\ref{thm:lower}, as follows. 

We denote $\Z/2=\{1,\sigma\}$ multiplicatively and $\Z$ additively, so an arbitrary element of $G$ is of the form $(n,1)$ or $(n,\sigma)$ for $n\in\Z$. Define subgroups
\[
A=\{(0,1),(1,\sigma)\},\qquad B=\{(0,1),(0,\sigma)\}
\]
of $G$, both isomorphic to $\Z/2$. The calculation
\[
(n,1)(1,\sigma)(-n,1)=(1-2n,\sigma)\neq (0,\sigma)
\]
verifies that the assumptions of Theorem~\ref{thm:lower} are satisfied for these subgroups, and clearly $A\times_Q B$ is isomorphic to $\Z/2$. Hence $\infty=\cld(\Z/2)\leq \TC_{\Z/2}(\Z)$.

This example is realized topologically by the bundle
\[
\xymatrix{
S^1 \ar[r] & S^1\times_{\Z/2} S^\infty \ar[r] & \R P^\infty
}
\]
associated to the universal principal $\Z/2$-bundle $S^\infty\to \R P^\infty$ with fibre $S^1\subset \C$ acted on by conjugation. 
\end{ex}
   
\section{An upper bound}

We now turn our attention to upper bounds for $\TC[\rho:G\tto Q]$, and prove a generalisation of \cite[Proposition 3.7]{Grant}. Recall that $\Delta:G\to G\times_Q G$ denotes the diagonal homomorphism.

\begin{thm}\label{thm:upper}
Let $\rho:G\tto Q$ be a group epimorphism with kernel $N$. Let $H\lhd G$ be a normal subgroup such that $[H,N]=1$. Then $\Delta(H)$ is normal in $G\times_Q G$, and
$$
\TC[\rho:G\tto Q]\le \cld\left(\frac{G\times_Q G}{\Delta(H)}\right).
$$
\end{thm}

\begin{proof}
The verification that $\Delta(H)$ is normal in $G\times_Q G$ is straightforward and left to the reader. Let $W$ denote the quotient. The diagram of groups
\[
\xymatrix{
 & & G\ar[d]^\Delta & & \\
1 \ar[r]  & H \ar@{^{(}->}[ur]\ar[r]^{\Delta|_H} & G\times_Q G \ar[r] & W \ar[r] & 1
}
\]
in which the row is an extension may be realised by a diagram of aspherical spaces
\[
\xymatrix{
& K(G,1) \ar[d]^-{f} & \\
K(H,1) \ar[r]^-{i} \ar[ur]  & K(G\times_Q G,1) \ar[r]^-{p} &  K(W,1)
}
\]
 in which the row is a fibration sequence. By Definition~\ref{def:pTChomo} and Lemma~\ref{lem:secatfibering}, we therefore have
 \[
 \TC[\rho:G\tto Q]=\secat(f)\leq \cat(K(W,1))=\cld(W).
 \]
 \end{proof}

\begin{cor}\label{cor:central}
If the extension
\[
\xymatrix{
1 \ar[r] & N \ar[r] & G \ar[r]^\rho & Q \ar[r] & 1
}
\]
is central, then $\TC[\rho:G\tto Q]=\cld(N)$.
\end{cor}

\begin{proof}
If $N$ is central in $G$ then we may take $H=G$ in Theorem ~\ref{thm:upper}. One checks that the map
\[
\frac{G\times_Q G}{\Delta(G)}\to N,\qquad [g,h]\mapsto gh^{-1}
\]
is a group isomorphism. Therefore $\TC[\rho:G\tto Q]\leq \cld(N)$. On the other hand we have
\[
\cld(N)=\TC(N)\leq \TC[\rho:G\tto Q],
\]
where the equality follows from the fact that $N$ is abelian and the inequality is Lemma~\ref{lem:fibre}.
\end{proof}

\section{Fadell--Neuwirth fibrations}

Recall that for a topological space $X$ and integer $m\ge1$, the \emph{$m$-th ordered configuration space} of $X$ is the space
\[
F(X,m):=\{(x_1,\ldots, x_m)\in X^m \mid i\neq j \implies x_i\neq x_j\},
\]
topologised as a subspace of the Cartesian power $X^m$. Fadell and Neuwirth \cite{FadellNeuwirth} showed that when $X$ is a manifold and $n\ge0$, the projection maps
\[
p:F(X,m+n)\to F(X,m),\qquad (x_1,\ldots , x_{m+n})\mapsto (x_1,\ldots , x_m)
\]
are locally trivial fibrations, and used this to study the homotopy type of configuration spaces. These so-called \emph{Fadell-Neuwirth fibrations} are of potential relevance in robotics, as calculations of their parametrised topological complexity give information about the motion planning problem for $n$ agents moving in the space $X$, avoiding collisions with each other and with $m$ obstacles whose positions may not be know in advance. For $X=\R^d$ the Euclidean space of dimension $d$, these calculations were carried out in the papers \cite{CFW1,CFW2}, where it was shown that 
\[
\TC[p:F(\R^d,m+n)\to F(\R^d,m)]=\begin{cases} 2n+m-1 & d\ge3\mbox{ odd}, \\ 2n+m-2 & d\ge2\mbox{ even.} \end{cases}
\]
The case $d$ even turns out to be significantly more difficult, with the cup-length calculation in \cite{CFW2} running to several pages. Here we note that for $d=2$ the spaces involved are all aspherical, and so we may apply the methods developed in this paper to give a shorter proof. 

We first recall some facts about duality groups. Recall that a group $\pi$ is a \emph{duality group of dimension $n$} if there exists a $\Z\pi$-module $C$ and an element $e\in H_n(\pi;C)$ such that cap product with $e$ induces an isomorphism
\[
-\cap e: H^k(\pi;A)\stackrel{\simeq}{\to} H_{n-k}(\pi;A\otimes C)
\]
for all $k$ and all $\Z\pi$-modules $A$. It follows that $\cld(\pi)=n$.

\begin{lem}[Bieri--Eckmann \cite{BE}]\label{lem:duality}
\begin{enumerate}
\item If $N$ and $Q$ are duality groups of dimensions $r$ and $s$ respectively, which fit into an extension of groups
\[
\xymatrix{
1 \ar[r] & N \ar[r] & \pi \ar[r] & Q\ar[r] & 1,
}
\]
then $\pi$ is a duality group of dimension $r+s$.
\item Non-trivial free groups are duality groups of dimension $1$.
\end{enumerate}
\end{lem} 

\begin{thm}[\cite{CFW2}]
Let $m\ge2$ and $n\ge1$, and let $p:F(\C,m+n)\to F(\C,m)$ denote the Fadell-Neuwirth fibration with fibre $F(\C_m,n)$, where $\C_m:=\C\setminus\{1,2,\ldots ,m\}$ is the $m$-th punctured plane. Then
\[
\TC[p:F(\C,m+n)\to F(\C,m)]=2n+m-2.
\]
\end{thm}

\begin{proof}
Let $P_r=\pi_1(F(\C,r))$ denote the pure braid group on $r$ strands. Then $p:F(\C,m+n)\to F(\C,m)$ realizes the epimorphism $\rho:P_{m+n}\to P_m$ which deletes the last $n$ strands. The kernel of $\rho$ is $\overline{P}_{n,m}:=\pi_1(F(\C_m,n))$, the $n$-strand braid group of the $m$-th punctured plane. All of these groups are iterated semi-direct products of free groups \cite{CS}, hence are duality groups of dimension equal to the number of free factors. In particular, $\cld(P_r)=r-1$ and $\cld(\overline{P}_{n,m})=n$.

Let $A$ and $B$ be the subgroups of $P_{m+n}$ arising in the proof of \cite[Proposition 3.3]{GLO}. Namely, $A$ is free abelian of rank $m+n-1$ generated by the braids $\alpha_j$ which pass the $j$-th strand over and behind the last $m+n-j$ strands and back to its starting position, for $j=1,2,\ldots,m+n-1$; and $B$ is the image of the embedding $P_{m+n-1}\hookrightarrow P_{m+n}$ as the first $m+n-1$ strands. It is shown in \cite{GLO} using linking numbers that $gAg^{-1}\cap B=\{1\}$ for all $g\in P_{m+n}$, so $A$ and $B$ satisfy the assumptions of Theorem \ref{thm:lower}.

The group $B$ fits in an extension
\[
\xymatrix{
1\ar[r] & \overline{P}_{n-1,m} \ar[r] & B \ar[r]^-{\rho|_B} & P_m \ar[r] & 1.
}
\]
Pulling back this extension along the map $\rho|_A: A\to P_{m}$ produces an extension  
\[
\xymatrix{
1\ar[r] & \overline{P}_{n-1,m} \ar[r] & A\times_{P_m} B \ar[r] & A \ar[r] & 1.
}
\]
Now, the groups $\overline{P}_{n-1,m}$ and $A$ are duality groups, of respective dimensions $n-1$ and $m+n-1$. Hence by Theorem~\ref{thm:lower} and Lemma~\ref{lem:duality},
\[
\TC[\rho:P_{m+n}\to P_m]\geq\cld(A\times_{P_m} B)=(n-1)+(m+n-1)=2n+m-2.
\]

Now let $\mathcal{Z}\leq P_{m+n}$ denote the centre, an infinite cyclic group. By Theorem \ref{thm:upper} we have
$$
\TC[\rho:P_{m+n}\to P_m]\leq \cld\left(\frac{P_{m+n}\times_{P_m} P_{m+n}}{\Delta(\mathcal{Z})}\right).
$$
Pulling back the extension
\[
\xymatrix{
1\ar[r] & \overline{P}_{n,m} \ar[r] & P_{m+n} \ar[r]^-{\rho} & P_m \ar[r] & 1
}
\]
by $\rho:P_{m+n}\to P_m$ gives an extension
\[
\xymatrix{
1 \ar[r] & \overline{P}_{n,m} \ar[r] & P_{m+n}\times_{P_m} P_{m+n} \ar[r] & P_{m+n} \ar[r] & 1.
}
\]
Taking the quotient by $\mathcal{Z}$ in each of the latter groups gives rise to an extension
 \[
\xymatrix{
1 \ar[r] & \overline{P}_{n,m} \ar[r] & \dfrac{P_{m+n}\times_{P_m} P_{m+n}}{\Delta(\mathcal{Z})} \ar[r] & \dfrac{P_{m+n}}{\mathcal{Z}} \ar[r] & 1
}
\]
in which the kernel and quotient are duality groups of dimensions $n$ and $m+n-2$ respectively (the latter follows from a splitting $P_{m+n}\cong \overline{P}_{2,m+n-2}\times\mathcal{Z}$). Hence by Lemma~\ref{lem:duality} the middle group is a duality group of dimension $2n+m-2$, which gives the desired upper bound.
\end{proof}


\begin{thebibliography}{99}

\bibitem{AS} M.\ Arkowitz, J.\ Strom, {\em The sectional category of a map},\/ Proc.\ Roy.\ Soc.\ Edinburgh Sect.\ A {\bf 134} (2004), no. 4, 639–652. 

\bibitem{BE} R.\ Bieri, B.\ Eckmann, {\em Groups with homological duality generalizing Poincar\'{e} duality},\/ Invent.
Math. 20 (1973), 103–124.

\bibitem{BCE} Z.\ B\l aszczyk, J.\ Carrasquel, A.\ Espinosa,{\em On the sectional category of subgroup inclusions and Adamson cohomology theory},\/ arXiv:2012.11912.


\bibitem{CFW1} D.\ C.\ Cohen, M.\ Farber, S.\ Weinberger, {\em Topology of parametrised motion planning algorithms},\/
 SIAM J.\ Appl.\ Algebra Geom. {\bf 5} (2021), no. 2, 229--249.
 
\bibitem{CFW2} D.\ C.\ Cohen, M.\ Farber, S.\ Weinberger, {\em Parametrised topological complexity of collision-free motion planning in the plane},\/ 
arXiv:2010.09809.

\bibitem{CS}  D.\ C.\ Cohen, A.\ I.\ Suciu,{\em  Homology of iterated semidirect products of free groups},\/ J.\ Pure Appl.\ Algebra {\bf 126} (1998), no. 1-3, 87–120.

%\bibitem{Dold} A.\ Dold, {\em Partitions of unity in the theory of fibrations},\/ Ann.\ of Math.\ (2) {\bf 78} (1963), 223–255.

\bibitem{Dranish} A.\ Dranishnikov, {\em On topological complexity of hyperbolic groups},\/ Proc.\ Amer.\ Math.\ Soc.\ {\bf 148} (2020), no. 10, 4547–4556.

\bibitem{EG} S. Eilenberg, T. Ganea, \emph{On the Lusternik-Schnirelmann category of abstract groups}, Ann. of Math. (2) {\bf 65} (1957), 517--518.

%\bibitem{Fadell} E.\ Fadell, {\em On fiber homotopy equivalence},\/ Duke Math.\ J.\ {\bf 26} (1959), 699–706.

\bibitem{FadellNeuwirth} E.\ Fadell, L.\ Neuwirth, {\em Configuration spaces},\/
Math.\ Scand.\ {\bf 10} (1962), 111-118.

\bibitem{Far} M.\ Farber,
\emph{Topological complexity of motion planning},\/
Discrete Comput.\ Geom.\ {\bf 29} (2003), 211-221.

\bibitem{FarSurvey} M. Farber, \emph{Topology of robot motion planning}, Morse theoretic methods in nonlinear analysis and
in symplectic topology, NATO Sci. Ser. II Math. Phys. Chem., vol. 217, Springer, Dordrecht,
2006, pp. 185--230. 

\bibitem{FGLO} M.\ Farber, M.\ Grant, G,\ Lupton, J.\ Oprea, {\em Bredon cohomology and robot motion planning},\/ Algebr.\ Geom.\ Topol.\ {\bf 19} (2019), no. 4, 2023–2059.

\bibitem{FarMes} M.\ Farber, S.\ Mescher, 
{\em On the topological complexity of aspherical spaces},\/ J.\ Topol.\ Anal.\ {\bf 12} (2020), no. 2, 293–319.

\bibitem{G-C} J.\ M.\ Garc\'{i}a Calcines, {\em Formal aspects on parametrised topological complexity and its pointed version},\/ arXiv:2103.10214

\bibitem{Grant} M.\ Grant, {\em Topological complexity, fibrations and symmetry}, Topology Appl. {\bf 159} (2012), 88—97.

\bibitem{GLO} M.\ Grant, G.\ Lupton, J.\ Oprea, {\em New lower bounds for the topological complexity of aspherical spaces}, Topology Appl. {\bf 189} (2015), 78—91.

\bibitem{May} J.\ P.\ May, {\em A concise course in algebraic topology},\/ Chicago Lectures in Mathematics. University of Chicago Press, Chicago, IL, 1999.


\bibitem{MayPonto} J.\ P.\ May, K.\ Ponto, {\em More concise algebraic topology. Localization, completion, and model categories},\/ Chicago Lectures in Mathematics. University of Chicago Press, Chicago, IL, 2012.

\bibitem{Schwarz} A. Schwarz, \emph{The genus of a fiber space}, A.M.S. Transl. \textbf{55} (1966), 49--140.

%\bibitem{Spanier} E.\ Spanier, {\em Algebraic Topology}, McGraw--Hill, 1966. 

\bibitem{Sta} J. R. Stallings, \emph{On torsion-free groups with infinitely many ends}\, Ann. of Math. (2) {\bf 88} (1968), 312--334.

\bibitem{Swa} R. Swan, \emph{Groups of cohomological dimension one}, J. Algebra {\bf 12} (1969), 585--610.
 

\end{thebibliography}
\end{document}